\documentclass[12pt,a4paper]{article}
\usepackage{amsmath,amscd,amsthm,amssymb,stmaryrd}
\usepackage{graphicx}
\usepackage{caption}

\newtheorem{theorem}{Theorem}[section]

\newtheorem{definition}[theorem]{Definition}
\newtheorem{lemma}[theorem]{Lemma}
\newtheorem{proposition}[theorem]{Proposition}
\newtheorem{remark}[theorem]{Remark}

\def\r{\mathbb{R}}

\title{The Dirichlet problem for the $\alpha$-translating soliton  equation on a strip}
\author{Rafael L\'opez\footnote{Partially
supported by MEC-FEDER
 grant no. MTM2017-89677-P}\\
 Departamento de Geometr\'{\i}a y Topolog\'{\i}a\\ Instituto de Matem\'aticas (IEMath-GR)\\
 Universidad de Granada\\
 18071 Granada, Spain\\
\texttt{rcamino@ugr.es}}

\date{}

\begin{document}
\maketitle

\begin{abstract}
We prove the existence of classical solutions to the Dirichlet problem for the $\alpha$-translating soliton equation defined in a strip of $\r^2$. We use the Perron method where a family of grim reapers are employed as barriers for solving the Dirichlet problem when the boundary data is formed by two copies of a convex function.  
\end{abstract}

\section{Introduction and statement of results}
Let $\Omega\subset\r^2$ be a smooth   domain and a given constant $\alpha>0$. We consider the   Dirichlet problem
\begin{eqnarray}
&&\mbox{div}\left(\dfrac{Du}{\sqrt{1+|Du|^2}}\right)=\left(\frac{1}{\sqrt{1+|Du|^2}}\right)^\alpha\quad \mbox{in $\Omega, u\in C^2(\overline{\Omega})$}\label{eq1}\\
&&u=\varphi\quad \mbox{on $\partial\Omega,$}\label{eq2}
\end{eqnarray}
where $D$ and div are the gradient and divergence operators and $\varphi$ is a continuous function in $\partial\Omega$. 
  Equation   (\ref{eq1}) is called the {\it $\alpha$-translating soliton equation} and the graph $\Sigma_u=\{(x,u(x)):x\in\Omega\}$ is an {\it $\alpha$-translating soliton} whose boundary is the graph   of $\varphi$.  In  the limit case  $\alpha=0$,   Equation (\ref{eq1}) is the known constant mean curvature equation. The motivation for the study of Equation (\ref{eq1}) comes from the case $\alpha=1$, where $1$-translating solitons, or simply, translating solitons,    appear  in the singularity theory of the mean curvature flow in $\r^3$ as the limit flow by a proper blow-up procedure near type II singular points: \cite{an,hs,il,wh}.  When $\alpha\not=1$, Equation (\ref{eq1}) extends to the case of the flow of surfaces by powers of the mean curvature:  \cite{sh1,sh2,sw}.

However, and besides of this interest in the mean curvature flow, 
 Equation (\ref{eq1}) already appeared in the classical article of   Serrin \cite{se} on the Dirichlet problem for quasilinear equations of divergence type. Possibly due to the extension of this article and its   focus on  the constant mean curvature equation, Equation (\ref{eq1}) seemed be forgotten. Indeed, in   \cite[p.477]{se}, Serrin considered Equation (\ref{eq1}) written as
\begin{equation}\label{ese}
(1+|Du|^2)^{\frac32}\mbox{div}\left(\dfrac{Du}{\sqrt{1+|Du|^2}}\right)=c(1+|Du|^2)^{\frac{3-\alpha}{2}},
\end{equation}
where $c>0$ is a constant and he studied the solvability of the Dirichlet problem for   mean domains. Recall that $\Omega$ is said to be mean convex if the mean curvature $H_{\partial\Omega}$ with respect to the inward normal is non-negative,  $H_{\partial\Omega}\geq 0$.  For arbitrary dimension, Serrin proved  the following result in \cite[p.478]{se}.

\begin{theorem}\label{t1} Let $\Omega$ be a bounded    domain in $\r^n$. Then there exists a unique solution of (\ref{eq1})-(\ref{eq2}) for any continuous function $\varphi$ if and only if:
\begin{enumerate}
\item $\Omega$ is mean convex when $\alpha\geq 1$.
\item $H_{\partial\Omega}>0$ when $0<\alpha<1$.
\end{enumerate} 
\end{theorem}

This result has been  recently revisited in the literature: see  \cite{ber,jj,jl,ma}. In the context of   translating solitons ($\alpha=1$), the existence of solutions of (\ref{eq1})   has been studied in \cite{aw,sx,wa}, including also  Neumann boundary conditions. In this article we study the solvability of (\ref{eq1})-(\ref{eq2}) when $\Omega$ is a strip of $\r^2$. The interest to the case that $\Omega$ is a strip comes by the result in \cite{wa} where Wang  found convex translating solitons that are graphs on a strip, although our methods and results in the present article differ on them. 

A first example of a translating soliton that is a graph on a strip is the grim reaper  $u(x,y)=-\log(\cos(y))$ where $u$ is defined in the strip $\Omega=\{(x,y)\in\r^2:-\pi/2<y<\pi/2\}$. Let us observe that $u\rightarrow\infty$ as $|y|\rightarrow\pm\pi/2$.  If we narrow the strip  to other strip included in $\Omega$, then   the restriction of $u$ to the new domain  has constant boundary values. For the case $\alpha>0$, it is possible to have the generalization of the grim reapers when we consider solutions of (\ref{eq1}) that depend only on one variable and these solutions will be candidates   to  be supersolutions of (\ref{eq1}).

In this paper we solve (\ref{eq1})-(\ref{eq2}) when the boundary data is formed by  two copies of a convex function of $\r$ extending what occurs for the grim reapers. Firstly, we need to introduce the next notation. Without loss of generality,   suppose that the strip is $\Omega_m=\{(x,y)\in\r^2:-m<y<m\}$, $m>0$. Let  $f$ be a smooth convex function defined in $\r$ and we extend $f$ to $\partial\Omega_m$ by defining $\varphi_f=\partial\Omega_m\rightarrow\r$ by $\varphi_f(x,\pm m)=f(x)$. 

\begin{theorem}\label{t2} Let $\Omega_m\subset\r^2$  be a strip. For each convex function $f$ defined in $\r$, there exists a   solution of (\ref{eq1}) for boundary values $u=\varphi_f$ on $\partial\Omega_m$ in the next cases:
\begin{enumerate}
\item For any $\alpha>1$ and $m>0$.
\item If $0<\alpha\leq 1$, provided the width of $\Omega_m$ satisfies $m <d(\alpha)$, where $d(\alpha)>0$ depends only on $\alpha$.
\end{enumerate}
\end{theorem}

This result was proved for the constant mean curvature equation ($\alpha=0$ in (\ref{eq1})) by Collin in \cite{co}.

 This paper is organized as follow. In Section \ref{sec2} we find the solutions of the one-dimensional case of (\ref{eq1}) and then in Section \ref{sec3} we recall  the maximum principle for Equation (\ref{eq1}) and some properties of the  solutions when $\Omega$ is a bounded domain. Finally in Section \ref{sec4}    we prove Theorem \ref{t2} by means of the Perron process of sub and supersolutions.

\section{The family of $\alpha$-grim reapers}\label{sec2}

In this section we generalize the grim reapers for every $\alpha>0$. Consider $\r^3$ the Euclidean space with canonical coordinates $(x,y,z)$. It is immediate that any translation of the space and  a rotation about an axis parallel to $e_3=(0,0,1)$ preserves the condition to be an $\alpha$-translating soliton.  In this section we study those $\alpha$-translating solitons that depend only on one variable, or in other words, the  surface is a cylindrical surface where all the rulings are parallel to a fixed direction. For our convenience, we consider firstly  solutions that are invariant in one direction orthogonal to $e_3$, namely, to the $x$-axis. The   surface is then generated by a curve $\gamma(s)=(y(s),z(s))$ contained in the $yz$-plane, which we suppose parametrized by the arc-length. Let $y'(s)=\cos\phi(s)$, $z'(s)=\sin\phi(s)$ for some angle function $\phi$. The  surface parametrizes by $X(x,s)=x(1,0,0)+(0,\gamma(s))=(x,y(s),z(s))$ and its  mean curvature is $H(x,s)=\phi'(s)/2$, where $\phi'(s)$ is the curvature of $\gamma(s)$. Definitively  the surface $X(x,s)$ is an $\alpha$-translating soliton if and only if $\gamma$ satisfies 
\begin{equation}\label{e1}
\begin{split}
&y'(s)=\cos\phi(s)\\
&z'(s)=\sin\phi(s)\\
&\phi'(s)=(\cos\phi(s))^\alpha.
\end{split}
\end{equation}
It is immediate that   a vertical line satisfies (\ref{e1}), indeed, $\gamma$ can be expressed as $\gamma(s)=(y_0,\pm s+z_0)$, $(y_0,z_0)\in\r^2$,  with $\phi(s)=\pm\pi/2$. After a translation in $\r^2$, we can suppose that $\gamma$ goes through the origin, hence $\gamma(0)=(0,0)$. 
 
\begin{proposition} \label{p-cl}
Let $\alpha>0$. There exists a family of solutions $\gamma=\gamma(s)$ of (\ref{e1}) satisfying 
\begin{equation}\label{e2}
y(0)=z(0)=0,\ \phi(0)=0,
\end{equation}
and with  the following properties:
\begin{enumerate}
\item $\gamma$ is a symmetric convex graph about the $y$-axis with one global minimum.
\item  Let $\gamma(y)=(y,z(y))$,   $z:(-d,d)\rightarrow\r$ with $d=d(\alpha)\leq\infty$ and $(-d,d)$ is the maximal domain of $z(y)$. Then we have: 
\begin{enumerate}
\item $d(\alpha)=\infty$ if $\alpha>1$.
\item $d(\alpha)<\infty$ if $0<\alpha\leq 1$ and $\gamma$ is asymptotic to the vertical lines $y=\pm d(\alpha)$.
\end{enumerate}
\end{enumerate}
\end{proposition}
\begin{proof}
Since the derivatives in (\ref{e1}) are bounded, then the maximal domain of the solution  $\gamma$ of (\ref{e1}) is $\r$. We prove that $\phi$ never attains the values $\pm\pi/2$.  If at some point $s=s_0$, it holds $\phi(s_0)= \pi/2$ (similar in case $-\pi/2$) and it is immediate that 
$$\bar{y}(s)=y(s_0),\ \bar{z}(s)=s-s_0+z(s_0),\ \bar{\phi}(s)=\pi/2$$
is a solution of (\ref{e1}) with the same initial condition at $s=s_0$ as the initial solution $\{y(s),z(s),\phi(s)\}$. By uniqueness, $\gamma(s)=(y(s_0),s-s_0+z(s_0))$, $\gamma$ is a straight line  and $\phi(s)=\pi/2$ for all $s\in\r$. This is a contradiction   because the initial condition at $s=0$ is $\phi(0)=0$. 

Therefore  $\gamma$ is a graph on the $y$-axis and we write $z=z(y)$. Taking into account that $\phi(s)\in (-\pi/2,\pi/2)$, we find that  $\phi'(s)>0$ and this implies that the (signed) curvature of $\gamma$ is positive proving that $z=z(y)$ is a convex function. On the other hand, Equations (\ref{e1})-(\ref{e2}) write now as 
\begin{equation}\label{hyper}
\frac{z''}{1+z'^2}=\frac{1}{(1+z'^2)^{\frac{\alpha-1}{2}}},\quad z(0)=0,\ z'(0)=0.
\end{equation}
Hence it is immediate that $z=z(y)$ is symmetric about $y=0$ with a global minimum at $y=0$. The integration of (\ref{hyper}) is given in terms of first hypergeometric function $_2F_1(a,b;c;x)$, where it is is known that  when $\alpha\in (0,1]$, the domain of $z=z(y)$ is a bounded (symmetric) interval $(-d,d)$, $d=d(\alpha)$ depending on $\alpha$ with $\lim_{y\rightarrow\pm d}z(y)=\pm\infty$ and if $\alpha>1$, then $z=z(y)$ is defined on the entire $y$-axis.
\end{proof}

\begin{remark} Some explicit solutions of (\ref{e1}) can be obtained by simple quadratures:
\begin{enumerate}
\item Case $\alpha=1$. Then $z(y)=-\log(\cos(y))$, $\gamma$ is the grim reaper and $d=\pi/2$.
\item Case $\alpha=2$. Then $z(y)=\cosh(y)$, $\gamma$ is the catenary and $d=\infty$.
\item Case $\alpha=3$. Then $z(y)=y^2/2$, $\gamma$ is the parabola and $d=\infty$.
\end{enumerate}
In the limit case $\alpha=0$, we have  $z(y)=-\sqrt{1-y^2}$, $\gamma$ is a halfcircle and $d=1$.
\end{remark}

Recall that each solution $\gamma$ given in Proposition \ref{p-cl} corresponds with an $\alpha$-translating soliton where the rulings are orthogonal to the vector $e_3$. In order to find supersolutions in the Perron process, we need to use   $\alpha$-translating solitons whose rulings are not necessary orthogonal to $e_3$.  Having this in mind, we consider the $\alpha$-translating solitons of Proposition \ref{p-cl} up to scaling and rotating about the $y$-axis.

\begin{definition}\label{d-22}
 Let $\alpha>0$. If $w=w(y)$ is a solution of (\ref{hyper}), we define the uniparametric family of $\alpha$-grim reapers $w_\theta=w_\theta(x,y)$ as  
$$w_\theta(x,y)=\frac{1}{(\cos\theta)^{\alpha+1}}w((\cos\theta)^\alpha y)+(\tan\theta)x+a,$$
where    $\theta\in (-\pi/2,\pi/2)$ and $a\in\r$. 
\end{definition}

As consequence of Proposition \ref{p-cl},  if $\alpha>1$ $w_\theta$ is defined in $\r^2$  and if $0<\alpha\leq 1$, then the domain of $w_\theta$ is the strip 
$$\Omega_{d,\theta}=\{(x,y)\in\r^2: -\frac{d}{(\cos\theta)^\alpha}<y<\frac{d}{(\cos\theta)^\alpha}\}.$$
 In particular, if $0\leq \theta_1<\theta_2$, it follows that  $\Omega_{\theta_1}\subset\Omega_{\theta_2}$ and thus the domain 
 \begin{equation}\label{eq-do}
 \Omega_d:=\Omega_{d,0}=\{(x,y)\in\r^2:-d<y<d\}
 \end{equation}
 is contained in   $\Omega_{d,\theta}$ for all $\theta\in (-\pi/2,\pi/2)$.

 \section{Some properties of the solutions of Equation (\ref{eq1})}\label{sec3}

 In this section we collect some properties of the solutions of   (\ref{eq1}) with a special interest in the control of $|u|$ and $|Du|$ when $\Omega$ is a bounded domain. Here we make use of explicit examples of $\alpha$-translating solitons to get a priori $C^0$ estimates. Firstly we need to recall that Equation (\ref{eq1}) satisfies a maximum principle which is a   consequence of the comparison principle (\cite[Th. 10.1]{gt}).

  \begin{proposition}[Touching principle] Let $\Sigma_i$  be two $\alpha$-translating solitons, $i=1,2$.   If $\Sigma_1$  and $\Sigma_2$ have  a common tangent interior point and $\Sigma_1$ lies above $\Sigma_2$ around $p$, then $\Sigma_1$ and $\Sigma_2$ coincide at an open set around $p$. The same holds if $p$ is a common boundary point and the tangent boundaries coincide at $p$.
\end{proposition}
 
As a direct application of the touching principle, {\it there do not exist compact $\alpha$-translating solitons} because if $\Sigma$ were a such surface, we can place a vertical plane $\Pi$ tangent to $\Sigma$ and leaving $\Sigma$ in one side of $\Pi$: this is impossible by the touching principle because both $\Pi$ and $\Sigma$ are $\alpha$-translating solitons.

Besides the $\alpha$-grim reapers defined in Section \ref{sec2}, other family of useful $\alpha$-translating solitons is formed by the surfaces of revolution about a vertical axis, or equivalently, the radial solutions of Equation (\ref{eq1}). There are two types of rotational $\alpha$-translating solitons. The first type are convex entire graphs on $\r^2$ and the second ones are of winglike-shape: we refer to the reader to \cite{css} when $\alpha=1$ and to \cite[Th. 1.1]{sw2}  for the general $\alpha>0$. We are interested in the first ones to find $C^0$ estimates of the solutions of (\ref{eq1}). After a horizontal translation, we suppose that the rotation axis is the $z$-axis. 

\begin{definition} For each $\alpha>0$, and up to a constant, there exists an entire radially symmetric strictly convex solution of (\ref{eq1}), namely,  $\mathbf{b}=\mathbf{b}(r)$, $r^2=x^2+y^2$, with a global minimum   in the $z$-axis. We call the corresponding graph   $\mathcal{B}$  as the $\alpha$-bowl soliton.
\end{definition}

Up to a constant, suppose that the minimum of $\mathbf{b}=\mathbf{b}(r)$ is the value $0$, that is, $\mathbf{b}(0)=0$, so  the origin of $\r^3$ is the minimum of $\mathcal{B}$. For $t>0$, we intersect $\mathcal{B}$ with the horizontal plane of equation $z=t$ obtaining a compact  cap $\mathcal{B}_R$ whose boundary is a circle of radius $R>0$ included in the plane $z=t$. Moreover, $R\rightarrow\infty$ as $t\rightarrow\infty$. As a conclusion,  we have proved that  for any $R>0$, there exists a radial solution of (\ref{eq1})-(\ref{eq2}) with $\varphi=0$ on $\partial\Omega$ and $\Omega$ is a disk of radius $R>0$.   

  \begin{proposition} \label{t31}
   Let $\Omega\subset\r^2$ be a bounded domain. If  $u$ is a solution of (\ref{eq1})-(\ref{eq2}), we have
  \begin{enumerate}
  \item The solution is unique.
  \item There exists $C=C(\Omega,\varphi)$ a constant depending only on $\varphi$ and $\Omega$ such that 
\begin{equation}\label{eh}
C\leq u\leq \max_{\partial\Omega}\varphi\quad \mbox{in $\Omega$}.
\end{equation}
  \item  The maximum of the gradient is attained at some boundary point: 
  $$\max_{\overline{\Omega}}|Du|=\max_{\partial\Omega}|Du|.$$
\end{enumerate}
\end{proposition} 
   
   \begin{proof}
  \begin{enumerate}
  \item The uniqueness of solutions holds because   the right hand side of (\ref{eq1}) is non-decreasing on $u$ (\cite[Th. 10.1]{gt}), 
  \item Since the the right hand side of (\ref{eq1}) is non-negative, then $\sup_\Omega u=\max_{\partial\Omega}u=\max_{\partial\Omega}\varphi$. The lower estimate for $u$ is obtained by using   $\alpha$-bowl solitons as comparison surfaces. Indeed, let us take    a round disc $D_R\subset\r^2$ of radius $R>0$ sufficiently big so $\overline{\Omega}\subset D_R$. Consider the $\alpha$-bowl soliton $\mathcal{B}_R$ given by $\mathbf{b}(r)$ defined in $D_R$ with $\mathbf{b}_{|\partial D_R}=0$ and denote by $b_m$ the minimum value of $\mathbf{b}$. We move down $\mathcal{B}_R$ sufficiently so $\Sigma_u$ lies above $\mathcal{B}_R$, that is, if $(x,y,z)\in\Sigma_u$, $(x,y,z')\in\mathcal{B}_R$, then $z>z'$. Then we move up $\mathcal{B}_R$ until the first touching point with $\Sigma_u$. If the first contact occurs at some interior point, then the touching principle implies   $\Sigma_u\subset \mathcal{B}_R$. In other case,    if $\Sigma_u\not\subset \mathcal{B}_R$, the first contact point occurs when $\mathcal{B}_R$ touches a boundary point of $\Sigma_u$. In both cases, we  conclude $b_m\leq u-\min_{\partial\Omega}\varphi$ and consequently   $C:= b_m+\min_{\partial\Omega}\varphi\leq u$.
  \item Equation (\ref{eq1}) can be expressed as     
\begin{equation}\label{eq4}
(1+|Du|^2)\Delta u-u_iu_ju_{ij}-(1+|Du|^2)^{\frac{3-\alpha}{2}}=0.
\end{equation}
where $u_i=\partial u/\partial x_i$, $i=1,2$, and we assume the summation convention of repeated indices.  Define the function  $v^k=u_k$, $k=1,2$ and we differentiate (\ref{eq4}) with respect to $x_k$, obtaining for each $k=1,2$,
 \begin{equation}\label{eq3}
 \left((1+|Du|^2)\delta_{ij}-u_iu_j\right)v_{ij}^k+2\left(u_i\Delta u-u_ju_{ij}-  \frac{3-\alpha}{2}u_i(1+|Du|^2)^\frac{1-\alpha}{2} \right)v_i^k=0,
 \end{equation}
 where $\delta_{ij}$ is the Kronecker delta. We deduce that $|v^k|$ and then $|Du|$ has not an interior maximum. In particular, if $u$ is a solution of (\ref{eq1}), the maximum of $|Du|$ on the compact set $\overline{\Omega}$ is attained at some boundary point,  proving the result. 
 
  \end{enumerate}
  \end{proof}
     
  \section{Proof of Theorem \ref{t2}}\label{sec4}
  
  In this section we prove Theorem \ref{t2} in a successive number of steps. First, we define the operator 
$$Q[u]= \mbox{div}\left(\dfrac{Du}{\sqrt{1+|Du|^2}}\right)-\left(\frac{1}{\sqrt{1+|Du|^2}}\right)^\alpha.$$
For   the Perron process we need to have a subsolution of (\ref{eq1})-(\ref{eq2}). In the next result, $f$ is not necessarily a convex function.

\begin{proposition}\label{p-min} Let $\Omega\subset\r^2$ be a strip. If $f$ is a continuous function defined in $\r$, then there exists a solution $v^0$ of the Dirichlet problem
\begin{equation}\label{em1}
\begin{split} &\mbox{div}\left(\dfrac{Du}{\sqrt{1+|Du|^2}}\right)=0 \quad \mbox{in $\Omega$}\\
&u=\varphi_f\quad \mbox{on $\partial\Omega$}
\end{split}
\end{equation}
with the property $f(x)<v^0(x,y)$ for all $(x,y)\in\Omega$.
\end{proposition}

 This result was proved in \cite{co} following ideas of  \cite{js}: see Remark 2 in \cite{co} where $v^0$ is called the maximal solution of (\ref{em1}). 
   
 We begin with the  classical Perron
method of sub and supersolutions for the Dirichlet problem (\ref{eq1})-(\ref{eq2}): see \cite[pp. 306-312]{ch}, \cite[Sec. 6.3]{gt}).   
 Let $u\in C^0(\overline{\Omega})$ be a continuous function and let $D$ be a closed round disk in $\Omega$. We denote by $\bar{u}\in C^2(D)$ the unique solution of the Dirichlet problem
 $$
 \left\{\begin{array}{ll}
 Q[\bar{u}]=0 & \mbox{ in $D$}\\
 \bar{u}=u & \mbox{ on $\partial D$}
 \end{array}\right.$$
whose existence is assured by Theorem \ref{t1} and its uniqueness by Proposition \ref{t31}. We extend $\bar{u}$ by continuity to $\Omega$ by defining   $M_D[u]$ in $\Omega$ as 
 $$M_D[u]=\left\{\begin{array}{ll} 
 \bar{u},& \mbox{ in } D\\
 u,&\mbox{ in }\Omega\setminus D.
 \end{array}\right.$$
 The function $u$ is said to be a {\it supersolution} in $\Omega$ is satisfies   $M_D[u]\leq u$ for every closed round disk $D$ in $\Omega$.  
 
{\it Example.}  For any domain $\Omega\subset\r^2$, the function $u=0$ in $\overline{\Omega}$ is a supersolution in $\Omega$. This is because if $D\subset \Omega$ is a closed round disk, then $\bar{u}<0$ since $Q[0]<Q[\bar{u}]=0$ and the maximum principle applies. Thus $M_D[u]\leq 0$. 
 
 Moreover, for each $p\in\Omega$, there exists a supersolution $u$ with $u(p)<0$. Indeed, let $D\subset\Omega$ be a closed round disk centered at $p$, which suppose to be the origin of $\r^2$. Let $\mathbf{b}=\mathbf{b}(r)$ be that $\alpha$-bowl soliton with $\mathbf{b}_{|\partial D}=0$. Then the function $u$ defined as $u=\mathbf{b}$ in $D$ and $u=0$ in $\overline{\Omega}\setminus D$ is a supersolution.

 \begin{definition} Let $u\in C^0(\overline{\Omega})$. We say $u$ is a superfunction relative to $f$ if $u$  is a supersolution in $\Omega$ and $f\leq u$ on $\partial \Omega$. Denote
 by $\mathcal{S}_f$   the class of all superfunctions relative to $f$, that is, 
 $$\mathcal{S}_f=\{u\in C^0(\overline{\Omega}):M_D[u]\leq u \mbox{ for every closed round disk $D\subset\Omega$}, f\leq u \mbox{ on $\partial\Omega$}\}.$$ 
 \end{definition}

 \begin{lemma}   The set $\mathcal{S}_f$  is not empty.
 \end{lemma}
 
 \begin{proof} We prove that $v^0\in  \mathcal{S}_f$, where  $v^0$ is  the minimal solution  given in Proposition \ref{p-min}. Let $D\subset\Omega$ be a closed round disk. Since $v^0$ is a minimal surface, then $Q[v^0]<0$ and because $\overline{v^0}=v^0$ in $\partial D$, the maximum principle implies $M_D[v^0]=\overline{v^0}\leq v^0$ in $D$. On the other hand, $v^0=f$ on $\partial\Omega$, proving definitively that $v^0\in \mathcal{S}_f$. 
 \end{proof}

We now give some  properties about superfunctions  whose proofs are straight-forward: in the case of the constant mean curvature equation, we refer \cite{lo}; in the context of $\alpha$-translating solitons, see \cite[Lems. 4.2--4.4]{jj}.  

 \begin{lemma}\label{p-p}
  \begin{enumerate}
 \item If $\{u_1,\ldots,u_n\}\subset  \mathcal{S}_f$, then $\min\{u_1,\ldots,u_n\}\in \mathcal{S}_f$.
 \item The operator $M_D$ is increasing in $\mathcal{S}_f$.
 \item If $u\in \mathcal{S}_f$ and  $D$ is a closed round disk in $\Omega$, then $M_D[u]\in\mathcal{S}_f$.
 \end{enumerate}
 \end{lemma}

  We begin with the proof of Theorem \ref{t2}.  We take $d(\alpha)$ the number given in Propositon \ref{p-cl} and let $f$ be a convex function on $\r$.

 Consider the family of $\alpha$-grim reaper $w_\theta$  (Definition \ref{d-22}) and recall that $w_\theta$ is defined in $\r^2$ if $\alpha>1$ or in the strip  $\Omega_{d,\theta}$. In particular, by the definition of $\Omega_d$ in (\ref{eq-do}), we find $\Omega_m\subset\Omega_d\subset \Omega_\theta$ for any $\theta$. Thus it makes sense to restrict $w_\theta$ to the strip $\Omega_m$ and  we keep the same notation for its restriction in $\Omega_m$.  Consequently $w_\theta$ is a linear function on $\partial\Omega_m$ and the boundary of $\Sigma_{w_\theta}$ consists of two parallel lines.

Consider the subfamily of     $\alpha$-grim reapers 
$$\mathcal{G}=\{w_\theta: w_\theta\leq f\mbox{ on $\partial\Omega_m$},\theta\in(-\pi/2,\pi/2)\}.$$ 
Two observations are needed to state:
  \begin{enumerate}
  \item The set $\mathcal{G}$ is not empty because $f$ is convex.
 \item Let $v^0$ be the minimal surface given in Proposition \ref{p-min} with $v^0=f$ on $\partial\Omega_m$. Then   $Q[v^0]<0$. Since $Q[w_\theta]=0$ for all $w_\theta\in\mathcal{G}$, the comparison principle asserts that $w_\theta<v^0$ in $\Omega_m$ for all $\theta\in (-\pi/2,\pi/2)$. This implies that $v^0$ will play the role of a subsolution for (\ref{eq1})-(\ref{eq2}). 
 
 \end{enumerate}
 
We are going to construct a solution of Equation (\ref{eq1})  between the $\alpha$-grim reapers of $\mathcal{G}$ and the minimal surface $v^0$. Let 
 $$\mathcal{S}_f^*=\{u\in \mathcal{S}_f: w_\theta\leq u\leq v^0,\, \forall w_\theta\in\mathcal{G}\}.$$
We point out that   $\mathcal{S}_f^*$ is not empty because $v^0\in  \mathcal{S}_f^*$.

\begin{lemma}\label{l-s}
 The set $\mathcal{S}_f^*$ is stable for the operator $M_D$, that is, if $u\in \mathcal{S}_f^*$, then $M_D[u]\in \mathcal{S}_f^*$.
\end{lemma}

\begin{proof}
Let $u\in \mathcal{S}_f^*$. We know by Lemma \ref{p-p} that $M_D[u]\in \mathcal{S}_f$. On the other hand, since $w_\theta\leq u\leq  v^0$ and because $M_D$ is increasing (Lemma \ref{p-p} again), then for every closed round disk $D$ in $\Omega_m$ it follows that $M_D[w_\theta]\leq M_D[u]\leq M_D[v^0]$. Finally $M_D[w_\theta]=w_\theta$, so $w_\theta\leq M_D[u]$, and since $v^0$ is a supersolution, we conclude $M_D[v^0]\leq v^0$, hence $M_D[u]\leq v^0$. This proves the result. 
\end{proof}

We are going to complete the proof of Theorem \ref{t2}.

\begin{proposition}[Perron process] \label{p-pe}
The function $u:\Omega_m\rightarrow\r$ given by 
$$v(x,y)=\inf\{u(x,y):u\in \mathcal{S}_f^*\}$$
 is a solution of (\ref{eq1}) with $v=\varphi_f$ on $\partial\Omega_m$.\end{proposition}

\begin{proof} The proof consists of two parts.

 Claim 1. {\it The function $v$ is a solution of Equation (\ref{eq1})}.

The proof is standard and here we follow \cite{gt}. Let $p\in\Omega_m$ be an arbitrary fixed point of $\Omega_m$.  Consider a sequence $\{u_n\}\subset \mathcal{S}_f^*$ such that $u_n(p)\rightarrow v(p)$ when $n\rightarrow\infty$. Let $D$ be  a closed round disk centered at $p$ and contained in $\Omega_m$. For each $n$, define the function 
$$v_n(q)=\min\{u_1(q),\ldots,u_n(q)\},\quad q\in\overline{\Omega_m}.$$
Then $v_n\in \mathcal{S}_f^*$ by Lemma \ref{p-p}. By the definition of $M_D$, we deduce  $M_D[v_n](p)\rightarrow v(p)$ as $n\rightarrow\infty$ (Lemma \ref{l-s}). Set $V_n=M_D[v_n]$. Then $\{V_n\}$ is a decreasing sequence bounded from below by $w_\theta$ for all $w_\theta\in\mathcal{G}$ and satisfying (\ref{eq1}) in the disk $D$. Consequently the functions $V_n$ are uniformly bounded on compact sets $K$ of $D$. In each compact set $K$, the norms of the gradients $|DV_n|$ are bounded by a constant depending only on $K$ and using H\"{o}lder estimates of Ladyzhenskaya and Ural'tseva, there exist uniform $C^{1,\beta}$ estimates for the sequence $\{V_n\}$ on $K$: see \cite{si} for interior  estimates for the mean curvature type equations in two variables  and   in the context of Equation (\ref{eq1}), the estimates were  proved in \cite{gjj}. By compactness, there exists a subsequence of $V_n$, that we denote $V_n$ again, such that $\{V_n\}$ converges on $K$ to a $C^2$ function $V$ in the $C^2$ topology and by continuity, $V$ satisfies (\ref{eq1}). Moreover, by construction, at the fixed point $p$ we have $V(p)=v(p)$.

It remains to prove that $V=v$ in $int(D)$, not only at the fixed point $p$. For $q\in int(D)$, we do a similar argument for $q$, so let $\{\tilde{u}_n\}\subset\mathcal{S}_f^*$, $\tilde{u}_n(q)\rightarrow v(q)$. Let $\tilde{v}_n=\min\{V_n,\tilde{u}_n\}$ and $\tilde{V}_n=M_D[\tilde{v}_n]$. Again $\tilde{V}_n$ converges on $D$ in the $C^2$ topology to a $C^2$ function $\tilde{V}$ satisfying (\ref{eq1}) and $\tilde{V}(q)=v(q)$. By construction, $\tilde{V}_n\leq \tilde{v}_n\leq V_n$, hence $\tilde{V}\leq V$.  Since $v\leq \tilde{V}$, we infer $\tilde{V}(p)=v(p)=V(p)$. Thus $V$ and $\tilde{V}$ coincide at an interior point of $D$, namely, the point $p$, and both functions $V$ and $\tilde{V}$ satisfy the $\alpha$-translating soliton equation (\ref{eq1}). Because $\tilde{V}\leq V$,  the touching principle implies  $V=\tilde{V}$ in $int(D)$. In particular, $V(q)=\tilde{V}(q)=v(q)$. This shows that $V=v$ in $int(D)$ and the claim is proved.

Claim 2. {\it The function $v$ is continuous up to $\partial\Omega_m$ with $v=\varphi_f$ on $\partial\Omega_m$}.

In order to finish the proof of Theorem \ref{t2}, we prove that the function $v$ takes   the value $f$ on $\partial\Omega_m$ and consequently $v$ is continuous up to $\partial\Omega_m$ proving that $v\in C^2(\Omega_m)\cap C^0(\overline{\Omega_m})$. Let us observe that the graph  of $\varphi_f$ consists of two copies of the graph of $f$, namely, 
$$\Gamma_{\varphi_f}=\Gamma_1\cup\Gamma_2=\{(x,m,f(x)):x\in\r\}\cup \{(x,-m,f(x)):x\in\r\}.$$
Let $p=(x_0,m)\in\partial\Omega_m$ be a boundary point of $\Omega_m$ (similar argument if $p=(x_0,-m)$). Because of the convexity of $f$, in the plane of equation $y=m$ the tangent line $L_p$ to the  planar curve $\Gamma_1$  leaves $\Gamma_1$ above $L_p$. Taking into account the  symmetry of $\varphi_f$ and the convexity of $f$, there exists an $\alpha$-grim reaper $w_\theta^p$ such that $w_\theta^p(p)=f(x_0)$ and $w_\theta^p<f$ in $\Gamma_{\varphi_f}\setminus\{(x_0,m,f(x_0),(x_0,-m,f(x_0))\}$, or in other words, $\partial\Sigma_{w_\theta^p}$ lies strictly below $\partial\Sigma_v$ except   at the points $(x_0,m,f(x_0)$ and $(x_0,-m,f(x_0))$, where $\Sigma_{w_\theta^p}$ and $\Sigma_v$ coincide.

Therefore the function $w_\theta^p$ and the minimal surface $v^0$ form a modulus of continuity in a neighbourhood of $p$, namely, $w_\theta^p\leq v\leq v^0$. Because  $w_\theta^p(p)=v^0(p)=f(p)$, we infer that  $v(p)=f(p)$ and this completes the proof of Theorem \ref{t2}.
 \end{proof}

\section*{Acknowledgements}
The author has been partially
supported by MEC-FEDER
 grant no. MTM2017-89677-P

\end{document}